\documentclass[12pt]{amsart}
\pdfoutput=1
\usepackage[margin=1in]{geometry}
\usepackage[utopia]{mathdesign}
\usepackage[mathscr]{euscript}
\usepackage{hyperref}

\hypersetup{
    colorlinks=true,
    linkcolor=black,
    citecolor=black,
    filecolor=black,
    urlcolor=black,
}

\usepackage{amsmath}
\usepackage{graphicx}
\usepackage{units}
\usepackage{verbatim}
\usepackage{tikz}
\usepackage{mathtools}
\usepackage[utf8]{inputenc}

\newtheorem{theorem}{Theorem}

\newtheorem{claim}{Claim}

\newcommand{\mb}{\mathbb}
\newcommand{\eps}{\varepsilon}

\begin{document}

\title{A Point in a $nd$-Polytope is the Barycenter of $n$ points in its $d$-Faces}

\author[Michael Gene Dobbins]{Michael Gene Dobbins \\ GaiA, Postech}


\begin{tikzpicture}[overlay, remember picture]
\path (current page.north west) ++(24pt,-24pt) node[below right] 
{\href{http://doi.org/10.1007/s00222-014-0523-2}{%
\vbox{\noindent\footnotesize
The final publication is available at springerlink.com\newline 
DOI~10.1007/s00222-014-0523-2}}}%
;
\end{tikzpicture}

\begin{abstract}
Using equivariant topology, we prove that it is always possible to find $n$ points in the $d$-dimensional faces of a $nd$-dimensional convex polytope $P$ so that their center of mass is a target point in $P$.  Equivalently, the $n$-fold Minkowski sum of a $nd$-polytope's $d$-skeleton is that polytope scaled by $n$.  This verifies a conjecture by Takeshi Tokuyama.
\end{abstract}


\maketitle

The goal of this article is to prove the following theorem, recently conjectured by Takeshi Tokuyama for the 1-skeleton (see Acknowledgments),
which to the author's knowledge, was the first time the question had been considered.  
The conjecture was originally motivated by the engineering problem of determining how counterweights can be attached to some 3-dimensional object to adjust its center of mass to reduce vibrations when the object moves  \cite{arakelian2005shaking}\cite{socg2014weight}. 

\begin{theorem}\label{thrm:barycenter}
For any $nd$-polytope $P$ and for any point $p\in P$,  there are points\footnote{%
Note the $p_i$ are not required to be distinct.  In particular, if $p$ is in a $d$-face, we may choose {$p = p_1 = \dots = p_n$}.}
$p_1,\dots,p_n$ in the $d$-skeleton $S$ of $P$ with barycenter $p = \nicefrac{1}{n}(p_1 +\dots+ p_n)$. 
Equivalently,
\[ nP = \underbrace{S +\cdots +S}_{n} \vspace{-8pt} \]
where ($+$) denotes Minkowski summation. 
\end{theorem}

Theorem \ref{thrm:barycenter} is loosely related to several results in topology and combinatorial geometry.  
One version of the Borsuk-Ulam Theorem states that any continuous map $\psi$ from the sphere $\mb{S}^d$ to $\mb{R}^d$ attains the same value at an antipodal pair $\exists x.\ \psi(-x) = \psi(x)$.  This has numerous applications, a collection of which is presented in \cite{matouvsek2003using}.  
The barycenter of antipodal points is the origin, so  
Theorem \ref{thrm:barycenter} for $n=2,d=1$ is just 
the Borsuk-Ulam Theorem for $d=1$ where $\psi$ is the radial distance of a convex polygon in polar coordinates. 
Gromov's Waist Theorem may be viewed as a further refinement of the Borsuk-Ulam Theorem that says a continuous map from a sphere to a Euclidean space of lower dimension is constant on a large subset of correspondingly higher dimension distributed widely around the sphere \cite{gromov2003isoperimetry}.   
Intuitively, we may view Theorem \ref{thrm:barycenter} for $n=2$ as saying such a subset for the radial distance of a polytope includes an antipodal pair of a correspondingly lower dimensional subspace, the $d$-skeleton. 
Munkholm gave a generalization of the Borsuk-Ulam Theorem to cyclic actions of prime order on the sphere 
that says an equivariant map must be constant on a large subspace of orbits in the quotient space \cite{munkholm1969borsuk}. 
We may view Theorem \ref{thrm:barycenter} for $n$ prime as being related in an analogous way. 
After seeing the proof, we will make this connection more explicit.

In another sense, Theorem \ref{thrm:barycenter} is analogous to Carathéodory's Theorem.  Carathéodory's Theorem states that the convex hull of a set $X \subset \mb{R}^d$ is the union of convex hulls of all subsets of $d+1$ points in $X$.  Letting $X$ be the 0-skeleton (vertices) of a polytope, this means any point in a $d$-polytope is a weighted barycenter of $d+1$ vertices.  Bárány and Karasev recently gave conditions on a set for its convex hull to be the union of convex hulls of a smaller number of points \cite{barany2012notes}.  Here, we similarly reduce the number of points, but we further require any point in $P$ to be the exact barycenter and we only consider $P$ as the convex hull of a skeleton. 
Theorem \ref{thrm:barycenter} also falls in the general problem of determining how a set can be expressed as a Minkowski sum of smaller sets.  
Such decompositions have connections to algebraic geometry \cite{beckwith2012minkowski}\cite{sturmfels1996grobner}.

There are several related questions we may consider.  Does Theorem \ref{thrm:barycenter} still hold when the points are restricted to be in skeletons of distinct dimension?  In the case of 2 points, the theorem extends to $(2d{+}1)$-dimensional polytopes where the points are in the $d$ and $(d{+}1)$-skeleton respectively \cite{socg2014weight}.  
The regular simplex, however, is a counterexample to generalizing the theorem for any other pair of skeletons.  
The question still remains open for more than 2 points. 
Does Theorem \ref{thrm:barycenter} hold for the weighted barycenter with fixed but unequal weights?  The theorem fails to generalize for the 1-skeleton of a triangular prism in $\mb{R}^3$ with weights $1,1,1+\eps$, but may generalize when more weighted points are allowed \cite{socg2014weight}.  Note that the theorem implies an analogous statement for weights and skeletons of distinct dimension that result from sequential partitions of $d$ into equal parts.  For example, any point $p$ in a 15-polytope can be expressed as $p= \nicefrac{1}{5}(p_1 +p_2 +p_3 +p_4) + \nicefrac{1}{15}(q_1 +q_2 +q_3)$ for $p_i$ in its 3-skeleton and $q_i$ in its 1-skeleton. 
Another important question that is not addressed here is how to actually compute these points.  

\begin{proof}[Proof of Theorem \ref{thrm:barycenter}]
Assume first that $n$ is prime. 
Assume also that $p$ is in the interior of $P$, and choose coordinates so $p = {0}$.
Let 
\[ X = \left\{ \left[
\begin{array}{ccc} 
x_{1,1} & \cdots & x_{1,n} \\
 \vdots & \ddots & \vdots \\
x_{nd,1} & \cdots & x_{nd,n} \\
\end{array} \right]
\in \mb{R}^{nd \times n} :
\left[ \begin{array}{c} x_{1,1} \\ \vdots \\ x_{nd,1} \end{array} \right] 
+
\dots
+
\left[ \begin{array}{c} x_{1,n} \\ \vdots \\ x_{nd,n} \end{array} \right] 
= {0} \right\} \]
and let $Q = P^n \cap X$. 
That is, $Q \subset X$ is the convex polytope consisting of all $n$-tuples of points in $P$ with barycenter at the origin. 
Note $Q$ has dimension $n^2d-nd$. 
For now, we assume that $X$ intersects $P^n$ generically.  That is, every face of $Q$ is of the form $F = (E_1 \times \cdots \times E_n) \cap X$ where each $E_i$ is a face of $P$ and $\dim F = \dim E_1 + \cdots +\dim E_n -nd$.  Later we will use a perturbation argument to deal with the non-generic case.

We define a map $\phi = (\phi_1, \dots, \phi_n) : \partial Q \to Y$ where $Y = \{y \in \mb{R}^n: y_1 +\cdots+ y_n=0\}$  as follows.
We first take a barycentric subdivision of $\partial Q$, define $\phi$ on the vertices of the subdivision, and then interpolate on each simplex.  For a proper face $F$ of $Q$, let $q_F$ be the barycenter of the vertices of $F$, and let $\phi_i$ be the difference of $\dim E_i$ from the expected dimension, 
\[ \phi_i(q_F) = \dim E_i - \nicefrac{1}{n} \sum_{i=1}^n \dim E_i.\]  
Note that $\sum_{i=1}^n \phi_i(q_F) = 0$, so $\phi(q_F) \in Y$, 
and if $F$ is a vertex of $Q$, then $\dim E_1 +\cdots+ \dim E_n = nd$, so $\phi_i(q_F) = {\dim E_i -d}$. 
Now extend $\phi_i$ to $\partial Q$ by linear interpolation from the vertices to the simplices of the barycentric subdivision. 

Our goal is to find a vertex $q$ of $Q$ such that $\phi(q)=0$.
If we have such a vertex ${q = (E_1 \times \cdots \times E_n) \cap X}$, then $\dim E_i = d$ for all $i \in [n]$, 
which implies the columns of $q$ are points in the $d$-skeleton of $P$ with barycenter $0$.  
We first find a point in a $(n{-}1)$-face of $Q$ where $\phi$ vanishes, and then from there find our way down to a vertex.

\begin{claim}\label{claim:face}
$\phi$ vanishes on some point $z$ in a $(n{-}1)$-face of $Q$.
\end{claim}

Let $G \simeq \mathbb{Z}_n$ act on $X$ by cyclically permuting columns and on $Y$ by cyclically permuting coordinates,
and observe that $\phi$ is $G$-equivariant, $\phi_i(gx)=\phi_{gi}(x)$ for $g \in G$.
Let $T$ be the $(n{-}1)$-skeleton of $Q$, 
and observe that $T$ is $(n{-}2)$-connected, since $Q$ is $(n{-}2)$-connected and attaching cells of dimension $n$ or greater does not change $(n{-}2)$-connectivity.  

Suppose Claim \ref{claim:face} failed. 
Then, the restriction $\phi: T \to Y$ would not vanish, so we could define a map 
$ \phi': T \to (\mb{S}^{n-2} \hookrightarrow Y)$ by normalizing the image $\phi'(x) = \frac{\phi(x)}{\| \phi(x) \|}$.
Since $n$ is prime, every $g \in G$ is a generator, except for the identity.  
This implies that the only fixed point $y \in Y$, $y_{gi} = y_i$, is $y=0$.
Hence $\mb{S}^{n-2}$ is fixed point free.  
But we see now, $\phi'$ would be a $G$-equivariant map from a $(n{-}2)$-connected space to a fixed point free space, which by Dold's Theorem is impossible \cite{dold1983simple}\cite{matouvsek2003using}.  Thus, the claim holds.

The point $z$ from Claim \ref{claim:face} is in some simplex of the barycentric subdivision of $\partial Q$ with vertices $q_0,\dots,q_{n-1}$ respectively defined by some chain of faces $F_0 \subset \cdots \subset F_{n-1}$.  Let
\[ z = \sum_{k=0}^{n-1} t_k q_k \quad \text{where} \quad t_k \geq 0,\ \sum_{k=0}^{n-1} t_k = 1. \]
Since $X$ intersects $P^n$ generically, there is a sequence $i_0, \dots, i_{n-1}$ such that for all $k \in [n-1]$, $\dim E_{k,i_k} = \dim E_{k-1,i_k} +1$ and for $i \neq i_k$, $\dim E_{k,i} = \dim E_{k-1,i}$ 
where $E_{k,i}$ are faces of $P$ and $F_{k} = {(E_{k,1} \times \cdots \times E_{k,n}) \cap X}$. 
In particular, there is some index $j$ where the dimension of the face does not increase, 
$j \in {[n] \setminus \{i_1,\dots,i_{n-1}\}}$ so $\dim E_{0,j} = \dots = \dim E_{n-1,j}$.  Hence, $\phi_j(q_k) = \phi_j(q_0) -\nicefrac{k}{n}$.  Since $\phi$ is defined on simplices by linear interpolation, we have
\[ \phi_j(z) = \sum_{k=0}^{n-1} t_k \phi_j(q_k) = \phi_j(q_0) -r = 0 \quad \text{where} \quad r = \sum_{k=1}^{n-1} \tfrac{t_k k}{n} \]
Observe that $0 \leq r < 1$, and since $\phi_j(q_0)$ is an integer, $r$ must also be an integer, so $r = 0$, 
which implies $t_k =0$ for $k \neq 0$.  Thus, $z = q_0$ is a vertex of $Q$, so the theorem holds under our assumptions.

So far we have only considered the case where $X$ intersects $P^n$ generically; now suppose it does not.
The polytope can be defined by a linear inequality $P = \{x : A x \leq {1}\}$ for some $A \in \mb{R}^{f \times nd}$ where $f$ is the number of facets of $P$.  
For $\eps \in \mb{R}^{f \times nd}$, let $P_\eps = \{x : (A +\eps)x \leq {1}\}$, 
let $S_\eps$ be the $d$-skeleton of $P_\eps$, and let $S = S_0$ be the the $d$-skeleton of $P$. 
For almost every $\eps$ sufficiently small, $X$ does intersect $P_\eps^n$ generically, and by the argument just given, there is some $z_\eps \in S_\eps^n \cap  X$. 
Since $S_\eps$ is bounded for $\eps$ small, there is some limit point 
$z \in \lim_{\eps \to 0} z_\eps \subset S^n \cap  X$ where $\lim_{\eps \to 0} z_\eps$ is the set of limits of all convergent sequences.

Similarly, if $p$ is on the boundary of $P$, then a sequence $p_\eps$ with $\lim_{\eps \to 0} p_\eps =p$ will give  such a sequence $z_\eps$ with $z \in \lim_{\eps \to 0} z_\eps$.

Now consider $n =n_1\cdots n_k$ where $n_i$ is prime. We will argue by induction on $k$.  The theorem holds trivially if $n=1$.  Suppose the theorem holds for a $nd$-polytope $P$ and $m = n_1 \cdots n_{k-1}$ points.  Then, we can find points $p_{1},\dots,p_{m}$ in the $n_kd$-skeleton of $P$ such that $p = \nicefrac{1}{m} \sum_{i=1}^{m} p_i$. 
Since the theorem holds for $n_k$, each $p_i$ can be expressed as the barycenter of points $p_{i,1},\dots,p_{i,n_k}$ in the $d$-skeleton of $P$, $p_i = \nicefrac{1}{n_k} \sum_{j=1}^{n_k} p_{i,k}$.
Together this gives $p = \nicefrac{1}{n} \sum_{i=1}^{m}\sum_{j=1}^{n_k} p_{i,j}$
Thus, the theorem holds for all $n$.
\end{proof}

The introduction noted a connection between Theorem \ref{thrm:barycenter} and a Borsuk-Ulam type Theorem for $\mb{Z}_n$-actions for $n$ prime.  This may be seen more explicitly in an alternate proof of Claim \ref{claim:face} using the Euler class. 

The Euler cohomology class is a topological invariant of a vector bundle that indicates how cycles in the base space intersect the kernel of a generic section. 
In particular, if for any $n$ prime, the $\mb{Z}_n$-Euler class of a rank $k$ vector bundle over a CW-complex is non-trivial, then any section must vanish at some point on a $k$-cell of the base space. 
A map $\psi : \mb{S}^j \to \mb{R}^k$ defines a section of the trivial bundle $\pi : \mb{S}^j \times \mb{R}^k \to \mb{S}^j$.  A $\mb{Z}_n$-equivariant map  
defines a section of a vector bundle on the orbits of the $\mb{Z}_n$-action, and  
Munkholm showed that the Euler class of this vector bundle is non-trivial if $j \geq (n-1)k$ 
{\cite[Proposition 3.1]{munkholm1969borsuk}.}  
This implies that the map must be constant on some orbit, which for $n=2$ and $j=k$ is the Borsuk-Ulam Theorem.
A nice informal introduction to the Euler class and equivariant cohomology on smooth manifolds was recently given in \cite{tu2011equivariant}.

Returning to Claim \ref{claim:face}, 
the function $\phi$ above defines a section of the trivial bundle on ${(\partial Q \times Y)} \simeq {(\mb{S}^{n^2d-nd-1} \times \mb{R}^{n-1})}$ and since $\phi$ is $G$-equivariant, $\phi / G$ defines a section of the bundle 
\[ \pi : (\partial Q \times Y)/G \to \partial Q/G, \quad \pi \{(xg,yg): g \in G \} = \{xg : g \in G\}.\]
The $\mb{Z}_n$-Euler class of this vector bundle is non-trivial, so $\phi / G$ must vanish on some $(n{-}1)$-cell of $\partial Q /G$, and this lifts to a $(n{-}1)$-face of $Q$ where $\phi$ vanishes. 
While this approach requires more details, seeing a basic overview may provide geometric intuition behind the use of Dold's Theorem.

\section*{Acknowledgments}

I would especially like to thank Otfried Cheong and Stefan Langerman for their insights, and Pavle Blagojevic for suggesting a shorter proof of Claim \ref{claim:face}.  I would also like to thank Helmut Alt, Sang Won Bae, Moritz Firsching, Andreas Holmsen, Ivan Izmestiev, Roman Karasev, Igor Rivin, Louis Theran, and Takeshi Tokuyama for helpful discussions. 

This research was made possible by the 2013 Korean Workshop on Computational Geometry in Kanazawa, Japan, where Professor Tokuyama communicated his conjecture.  
This research was supported by NRF grant 2011-0030044 (SRC-GAIA) funded by the government of Korea.

\bibliographystyle{plain}
\bibliography{barycenter}

\end{document}